\documentclass[12pt]{amsart}
\usepackage[dvipdfmx]{graphicx, color}
\usepackage[all]{xy}
\usepackage{amssymb}
\usepackage{mathtools}
\usepackage{setspace}
\usepackage{amsfonts}
\usepackage{float}
\usepackage{amsmath}
\usepackage{amsthm}
\usepackage{tikz}

\usepackage{url}
\usepackage{bm}
\numberwithin{equation}{section}

\usepackage[super]{nth}

\calclayout

\usepackage{colortbl}
\usepackage{color}

\def\qed{\hfill $\Box$}

\newcommand{\N}{\mathbb{N}}
\newcommand{\Q}{\mathbb{Q}}
\newcommand{\R}{\mathbb{R}}

\newcommand\jeden {1\hskip-3.5pt1}

\def\qed{\hfill $\Box$}

\usepackage{xcolor}

\def\b0{\mbox{\boldmath $0$}}

\def\I{{\rm I}}
\def\II{{\rm I\hspace{-.1em}I}}

\newtheorem{thm}{\bf Theorem}[section]

\newtheorem{lem}[thm]{\bf Lemma}
\newtheorem{prop}[thm]{\bf Proposition}
\newtheorem{dfn}[thm]{\bf Definition}

\newtheorem{rem}[thm]{\bf Remark}

\begin{document}

\setlength{\baselineskip}{15pt}

\title[Fractal Betti numbers]
{An extension of fractal Euler number via persistent homology}
\author[K.~Nishijima]{Kosuke Nishijima}
\address[K.~Nishijima]{
Department of Pure and Applied Mathematics, Graduate School of Fundamental Science and Engineering, Waseda University, Tokyo 169-8555, Japan}
\email{k\_nishijima@akane.waseda.jp}
\dedicatory{}

\begin{abstract} 
In the context of geometric measure theory, Llorente-Winter introduced the (average) fractal Euler number as a notion of the Euler characteristic for fractals embedded in Euclidean space. However, the class of fractals to which it is applicable remains very limited. In the present paper, we modify this notion by  applying perspectives of persistent homology and partly the theory of magnitude, which have recently come from applied topology and category theory. We then demonstrate concrete calculation of our {\em average ph-fractal Euler number} for some classically well-known fractals, especially  the Cantor dust and Menger sponge which are excluded from Llorente-Winter's approach. 
\end{abstract}
\maketitle

\section{Introduction}
To extract qualitative information about the shape of a compact subset $X$ of Euclidean space $\R^d$, 
we would first look at its dimension and the number of connected components, or more generally, the number of  `holes'. As a typically nice case, if $X$ has a retraction to a finite cell complex, the Euler characteristics is defined as the alternative sum 
$$\chi(X)=\sum_{i \ge 0} (-1)^i \, b_i(X)$$
where $b_i(X)$ is the $i$-th Betti number, that is $\dim H_i(X; \R)$, intuitively understood as the number of `$i$-dimensional holes' inside $X$. 
However, if $X$ is a much more complicated space, like a fractal, it often happens that the dimension becomes a non-integer real number (in the sense of various fractal dimensions) and some Betti numbers do not exist. 
In fact, a fractal $X$ often admits {\em self-similarity}: 
there exists a finite number of affine contraction self-maps $f_1, \cdots, f_n: X \to X$ such that 
$$X=\bigcup_{k=1}^n f_k(X).$$
If there is a non-trivial $i$-dimensional cycle $\gamma$ of $X$, $f_k(\gamma)$ may not be trivial and infinitely many non-trivial cycles may possibly be created by iteration of the maps, then $b_i(X)=+\infty$. 

As a preceding approach in geometric measure theory, Llorente-Winter \cite{LW07} (also Winter and Z\"{a}hle \cite{W06, WZ13}) introduced a notion of Euler characteristics for such fractals, under a strong assumption. 
Let $X$ be a compact subset of $\R^d$. 
Their first step is to fatten the space -- consider closed $\epsilon$-neighborhoods of $X$ (also called $\epsilon$-parallel sets) 
$$X_\epsilon:= \{\; x \in \R^d \; | \; \inf_{a \in X} d(x, a) \le \epsilon\; \}$$
where $d(\cdot, \cdot)$ means the Euclidean metric. 
Now we assume that $\chi(X_\epsilon)$ exists for any $\epsilon > 0$ and there is the infimum of $s\ge 0$ satisfying that $\epsilon^s |\chi(X_\epsilon)|$ is bounded as $\epsilon \searrow 0$, called the {\em Euler exponent of $X$}, denoted by $\sigma=\sigma(X)$. 
Under this condition, they introduce the {\em fractal Euler number} 
\begin{equation}\label{fE}
\chi_{f}(X):=\lim_{\epsilon \searrow 0}  \left(\frac{\epsilon}{d_\infty}\right)^\sigma \chi(X_\epsilon)
\end{equation}
and the {\em average fractal Euler number} 
\begin{equation}\label{avfE}
\chi_{f}^{a}(X):=\lim_{\delta \searrow 0} \frac{1}{|\log \delta|} \int_\delta^1 \left(\frac{\epsilon}{d_\infty}\right)^\sigma \chi(X_\epsilon) \frac{d\epsilon}{\epsilon}
\end{equation}
where $d_\infty$ means the diameter of $X$. Namely,  $\sigma$ is the growth rate of $\chi(X_\delta)$ as $\delta \searrow 0$ and 
$\chi_{f}(X)$ is the proportional constant to $(\delta/d_\infty)^{- \sigma}$ in asymptotic expansion; if  the quantity exists, then $\chi_{f}^{a}(X)$ automatically converges to the same value. However, this is quite rare, that is discussed in detail in \cite{LW07}. 
Notice that those two notions of `Euler characteristics' represent a {\em relative metric feature of the embedding of $X$ in Euclidean space $\R^d$}; they are obviously invariant under Euclidean motions, but not topological nor fractal intrinsic feature as a metric space $X$ itself, and even not invariant under affine transformations of $\R^d$. 

In the present paper, we focus on the latter averaged quantity, which measures the limit behavior of proper Riemannian integral in the right hand side of (\ref{avfE}) as $\delta \searrow 0$. The existence of $\chi_{f}^{a}(X)$ is still a subtle issue  \cite{LW07}. 
Moreover, suppose that $X$ is self-similar and there is a non-zero `bad radius' $\epsilon>0$ so that $|\chi(X_\epsilon)|=+\infty$. Indeed, 
those are often satisfied even by classically well-known fractals. Then there are infinite many such bad radii and the integral in (\ref{avfE}) does no longer make sense for any $\sigma$. 
To overcome this trouble, our basic idea is to modify the meaning of the integral, inspired by new concepts from applied topology and category theory.

To make it precise, we employ the persistent homology $PH_*(X_\bullet)$  of the filtered family $X_\bullet=\{X_\epsilon\}$ indexed by  
$\{\, 0\le \epsilon \le d_\infty \, \mid \,  \mbox{$b_i(X_\epsilon)$ exists for $i=0,1,\dots, d$}\, \}$.  The $i$-th Betti function $b_i(X_\epsilon)$ is approximated by a finite step function summing up the characteristic functions over barcodes, that is $\sum_{|e|\ge \delta} \jeden_\epsilon$, where the sum runs over all barcodes $e \in PH_i(X_\bullet)$ with the lifetime $|e|=d_e-b_e \ge \delta$ ($b_e$ and $d_e$ are the birth and death time of $e$). Note that, by definition,  $\jeden_e(\epsilon)=1$ if $\epsilon \in e$, $0$ otherwise, so $\int_0^{d_\infty} \epsilon^{\sigma-1} \jeden_e d\epsilon
=\frac{1}{\sigma}\left(d_e^{\sigma}-b_e^{\sigma}\right)$ (for any $\sigma>0$). 
In case that  there is some non zero bad radius,  the Euler exponent $\sigma$ does not exist, so we instead adopt a new kind of fractal dimension, the {\em $i$-th PH-complexity} $\sigma_i=\sigma_i(X)$, defined by MacPherson-Schweinhart \cite{MS12} and Schweinhart \cite{S21}; roughly, $\sigma_i$ is the growth rate of the number of barcodes in $PH_i(X_\bullet)$ with $|e| \ge \delta$ as $\delta \searrow 0$. 
Then we replace the integral in the right hand side of (\ref{avfE}) by the finite sum 
\begin{equation}\label{S}
S_\delta (X):=
\sum_{i=0}^d \frac{(-1)^i}{\sigma_i} 
\sum_{|e|\ge \delta} \; \left(\left(\frac{d_e}{d_\infty}\right)^{\sigma_i}-\left(\frac{b_e}{d_\infty}\right)^{\sigma_i}\right)
\end{equation}
and define a new quantity, the {\em average ph-fractal Euler number}, by 
\begin{equation}\label{phfS}
\chi_{a}^{\rm phf}(X):= \lim_{\delta \searrow 0} \frac{1}{|\log \delta|} S_\delta(X). 
\end{equation}
Notice that the role of the parameter $\delta$ has been changed to control the lifetime, not just the end point of the region of integration. 
We also emphasize that this sum (\ref{S}) can be read off as (an alternative sum of) the {\em magnitude of persistent modules} in the sense of Govc-Hepworth \cite{GH21} after taking a suitable change of variables. See \S 2 and \S 3 for our precise definition and detail description. 

Apart from the technical matter of using persistent homology, our Euler number $\chi_{a}^{\rm phf}(X)$ essentially retains the same meaning as Llorente-Winter's $\chi_{f}^{a}(X)$,  but it is applicable to a wider class of fractals.  
In fact, we show that $\chi_{a}^{\rm phf}(X)$ coincides with $\chi_{f}^{a}(X)$ under suitable conditions (Proposition \ref{prp: the relation between phf Euler number and average fractal Euler number}).   
Furthermore, we demonstrate several precise computation of our Euler number $\chi_{a}^{\rm phf}(X)$ on classically famous fractals, where 
all arguments in computation are elementary. 
In \S 4, we treat the 3-adic Cantor set on the line and  the Sierpi\'nski carpet in the plane, which are also found in Llerente-Winter \cite{LW07}, and in \S 5, we take the Cantor dust on the plane and the Menger sponge in $3$-space (Figure 1), which are excluded from Llerente-Winter's approach because there are non-zero bad radii in each case. 
In the latter two examples, we observe a new phenomenon, called {\em little barcode dust}, which seemingly has not been recognized so far, and we accurately compute the contribution to our Euler number. 

Finally in \S 6, we give some conclusion together with supplementary notes. 
In particular,  instead of considering the parallel sets $X_\epsilon$, 
it would be very interesting to take finite samples ${\bf x} \subset X$ and the persistent homology $PH_*({\bf x}_\bullet)$ of  \^{C}ech complexes. In fact, Schweinhart \cite{S21} has introduced the {\em $PH_i$-dimension} through finite sampling as a different fractal dimension from the $i$-th PH-complexity. 
It is possible to improve the definition of our {\em average ph-fractal Betti numbers} by taking finite samples and the suprimum, that would be expected to relate to magnitude homology. We will discuss it elsewhere.

This paper is part of the author's master thesis \cite{Nishijima}. 
The author expresses his sincere gratitude to his advisor Prof. Toru Ohmoto, and also Dr. Chong Zheng, for introducing this subject  and engaging in numerous discussions.

 \begin{figure}[h]
 \centering
\includegraphics[height=6cm, pagebox=cropbox]{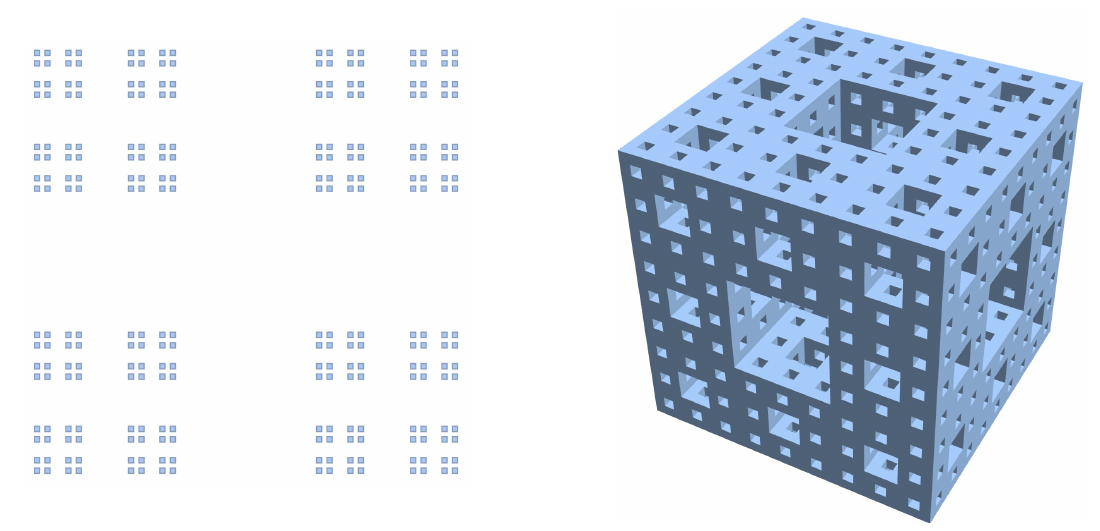}\\
\caption{Cantor dust and Menger sponge: their Euler numbers are computed as 
$\chi_{a}^{\rm phf}(C\times C)=0.1018\dots$ and $\chi_{a}^{\rm phf}(M)= -0.0001353\dots$, while $\chi_{f}^{a}$ is not defined for them.}
\label{CM}
 \end{figure}

\section{Persistent homology and PH-complexity}
\subsection{Persistent homology}
For simplicity, we first introduce the persistent homology for a filtration of topological spaces: 
\begin{equation*}
    X_\bullet \colon X_1 \subset X_2 \subset \dots .
\end{equation*} 
\indent
Let $H_i (X_j)$ denote the $i$-th homology group of $X_j$ with coefficients in a field $\mathbf{k}$ and set
\begin{equation*}
    PH_i (X_\bullet) \coloneqq \bigoplus_{j\in \N} H_i(X_j)
\end{equation*}
equipped with the action $x\cdot (c_1, c_2, \dots) = (0, c_1, c_2, \dots)$. 
Then, $PH_i(X_\bullet)$ naturally becomes a $\mathbf{k} [x] $-module. 
If it is finitely generated, then a generater is called a barcode, denoted it by $[b,d]$ where $b$ is the birth time and $d$ is the death time (or $\infty$).

\indent
This is more generally stated as follows. 
\begin{dfn}
    Let $R$ be a totally orderd set and  $\left\{V_r\right\}_{r\in R}$ be a family of vector space indexed by $R$. 
    Moreover, for $s,r\in R$ with $r\leq s$, there is a linear map $\rho_{s,r}\colon V_r \rightarrow V_s$. 
    Then, \textbf{persistent module} $\mathbf{V}$ on $R$ is a pair $\left(\left\{V_r\right\}_{r\in R}, \left\{\rho_{s,r}\right\}_{r,s\in R} \right)$ satisfying $2$ conditions:
    \begin{enumerate}
        \item $\rho_{r,r} = id_{V_r}$
        \item For $r,s,t \in R$ with $r\leq s \leq t$, $\rho_{t,r} = \rho_{t,s} \circ \rho_{s,r}$
    \end{enumerate}
\end{dfn}
In categorical terms, a persistent module is a functor from $R$, viewed as a category via its natural ordering, to the category of vector spaces.
Moreover, William \cite{W15} established the barcode decomposition theorem for a persistent module. 
A persistent module is {\em pointwise finite-dimensional} if for any $V_r$ is finite dimensional. 
Assume that $R$ a subset of the real line $\R$. 
For an interval $I \subset R$, the {\em interval module} $\mathbf{V} = \mathbf{k}_I$ is given by $V_t = \mathbf{k}$ for $t \in I$, $V_t = 0$ for $t \notin I$ and $\rho_{t,s} = id_{\mathbf{k}}$ for $s,t \in I$ with $s\leq t$. 
The interval $I$ is often called a barcode. 

\begin{thm}{\rm (\cite{W15})}
    Any pointwise finite-dimensional persistence module is a direct sum of interval modules.
\end{thm}

Let $X$ be a compact subset of a metric space and $X_\epsilon$ denotes of closed $\epsilon$-neighborhood of $X$ as mentioned in Introduction. 
Then, we define the $i$-th \textbf{bad radius set} of $X$ $\mathfrak{B}_i (X)$ by
\begin{equation*}
    \mathfrak{B}_i (X) \coloneqq \left\{ \epsilon \in \mathbb{R}_{\geq 0}\colon \dim H_i (X_\epsilon) = \infty \right\}.
\end{equation*} 
Put $\mathfrak{B}(X) = \bigcup \mathfrak{B}_i (X)$. 
For example, in case of the Cantor dust $X = C \times C$, the $\epsilon$-neighborhood filtration is as illustrated in Figure \ref{fig: filtration of Cantor dust}. 
\begin{figure}[bt]
    \centering
    \includegraphics[width=\textwidth]{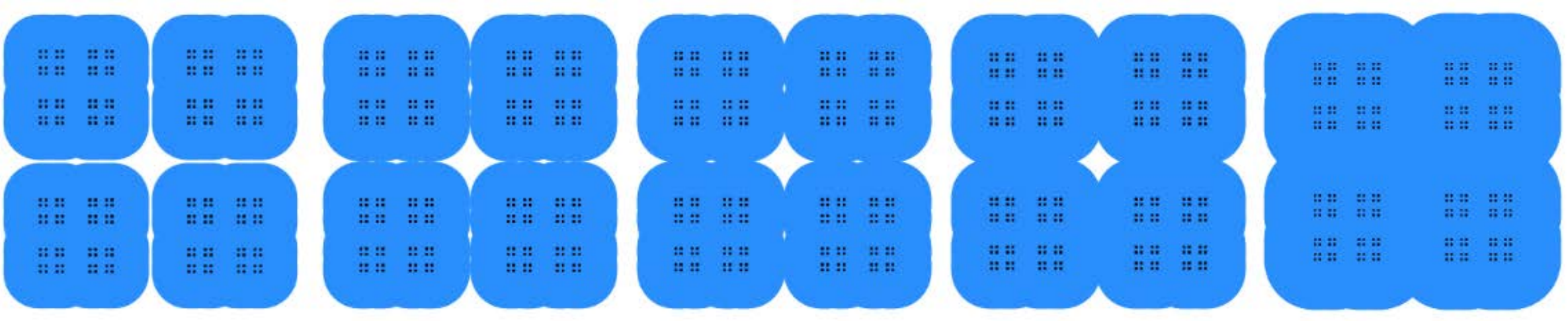}
    \caption{The $\epsilon$-filtration of the Cantor dust in $\R^2$ around $\epsilon = \frac{1}{6}$ (center of the figure)}
    \label{fig: filtration of Cantor dust}
\end{figure}
Figure \ref{fig: little barcode dust} is a magnified view of a part of $X_\epsilon$ for $\epsilon = \frac{1}{6}$, that shows that $\dim H_1(X_{\frac{1}{6}}) =\infty$. 
In the meaning of persistent homology, the infinitely many very small holes can be interpreted as generators with extremely short lifetimes. 
We refer to this phenomenon as \textbf{little barcode dust}. 
\begin{figure}[bt]
    \centering
    \includegraphics[width=\textwidth]{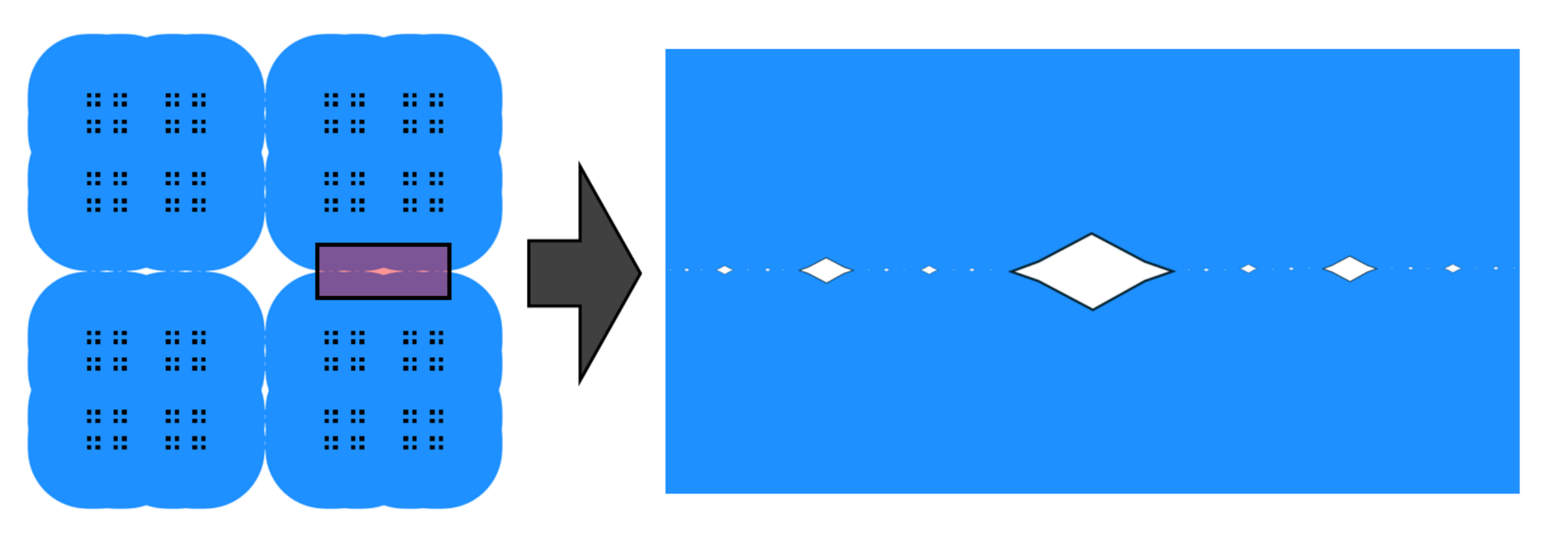}
    \caption{The little barcode dust of Cantor dust in $\R^2$ with $\epsilon = \frac{1}{6}$}
    \label{fig: little barcode dust}
\end{figure}

\indent
As the $i$-th persistent homology of $X$, we consider the $i$-th persistent homology defined with respect to the filtration 
\begin{equation*}
    X_\bullet = \left\{X_\epsilon\right\}_{\epsilon\in R}
\end{equation*}
indexed by the set $R = \mathbb{R}_{\geq 0} - \mathfrak{B}(X)$. 
Then $PH_i (X_\bullet)$ is pointwise finite-dimensional, thus $PH_i (X_\bullet)$ can be decomposed into a direct sum of barcodes. 

\subsection{PH-complexity}
To characterize the structure of fractals through the growth rate of the number of persistent homology barcodes, the following quantity was defined by MacPherson-Schweinhart \cite{MS12}. 
Here, we adopt the definition given by Schweinhart \cite{S21}.\\
\indent
Let $b(e)$ and $d(e)$ be the birth time and death time of a barcode $e\in PH_i(X_\bullet)$, respectively. 
Then, for $\epsilon > 0$, we define the function representing the number of barcodes $e$ with lifetime $|e| = d(e) - b(e)$ greater than $\epsilon$ as
\begin{equation*}
    I_{i,\epsilon}(X) = \# \left\{ \, e \in PH_i(X_\bullet) \colon |e| > \epsilon \, \right\}.
\end{equation*}
Then, MacPherson-Schweinhart\cite{MS12} defined the PH-complexity as the growth rate of the number of barcodes as the lifetime decreases, as follows:
\begin{dfn}
    The $i$-th PH-complexity $\mathrm{comp}_{PH_i}(X)$ of a compact subset $X$ of a metric space is defined as
    \begin{equation*}
        \mathrm{comp}_{PH_i} (X) = \sup \left\{\, c \, \colon \lim_{\epsilon \searrow 0} \epsilon^c I_{i,\epsilon} (X) = \infty \, \right\}.
    \end{equation*}
    \label{dfn:ph-complexity}
\end{dfn}
This perspective of focusing on the lifetime of the barcodes of $PH_i (X_\bullet)$, rather than observing $b_i (X_\epsilon)$, is the key to modify the average fractal Euler number \cite{LW07}. 
Furthermore, Schweinhart gives an equivalent definition of PH-complexity in \cite{S21}.
\begin{prop}
    \label{prp: equivalent definition of PH-complexity}
    (\cite{S21}, Appendix)
    Let $X$ be a compact subset of a metric space. Then
    \begin{equation*}
        comp_{PH_i} (X) = \inf \left\{ \, \alpha \, \colon \sum_{e\in PH_i (X_\bullet)} |e|^\alpha < \infty \, \right\}
    \end{equation*}
\end{prop}

As a preceding work, Venessa Robins considered the growth rate of Betti number $b_i(X_\epsilon)$ as $\epsilon \searrow 0$ in her thesis \cite{Robins}, while PH-complexity is defined as the growth rate of the number of the persistent homology barcodes. 
Note that PH-complexity can be defined even for the case of $\mathfrak{B}_i(X) \neq \{0\}$.
In particular, when $\mathfrak{B}_i(X)=\{0\}$ and $b(e)=0$ for any $e \in PH_i(X_\bullet)$, Robins' growth rates coincide with the PH-complexity. 
After MacPherson-Schweinhart \cite{MS12}, Schweinhart \cite{S21} intoroduces another fractal dimension, to which we will refer in the last section.

\section{Average $ph$-fractal Betti number}
We modify the concept of the average fractal Euler number introduced by Llerente-Winter\cite{LW07} to handle a broader class of objects, particularly those for which the bad radius set $\mathfrak{B}(X)$ is an at most countable set. 
In doing so, we decompose the Euler number in Equation \ref{avfE} into Betti numbers. 
Let $X$ be a compact subset of the $d$-dimensional Euclidean space. 
Let $PH_i(X_\bullet)$ denote the $i$-th persistent homology of $X_\bullet$ with coefficients in a field $\mathbf{k}$, $d_\infty$ the diameter of $X$ and $\sigma_i = \mathrm{comp}_{PH_i}(X)$.
We consider the integral of $\epsilon^{\sigma_i -1} b_i(X_\epsilon) = \sum_{e} \epsilon^{\sigma_i -1} \jeden_e$ with $|e|>\delta$. 
That is written as follows: 
\begin{dfn}
    For $i$ such that $\sigma_i \neq 0$, we define
    \begin{equation*}
        S_\delta^{(i)}(X) = \frac{1}{\sigma_i} \sum_{\substack{e \in PH_i(X_\bullet) \\ |e| >\delta}} \left\{ \left( \frac{d(e)}{d_\infty} \right)^{\sigma_i} - \left( \frac{b(e)}{d_\infty}\right)^{\sigma_i} \right\}.
    \end{equation*}
    For $\sigma_i = 0$, we set $S_\delta^{(i)}(X) = 0$.  
\end{dfn}

\begin{rem}
    We apply a change of variables:
    \begin{equation*}
        h(\epsilon) = \sigma_i \log \frac{d_\infty}{\epsilon} + \log \sigma_i \ \text{ for } \epsilon>0.
    \end{equation*}
    If $\epsilon = 0$, we put $h(\epsilon) = \infty$. 
    For every borcode $e$, set $b'(e) \coloneqq h(b(e))$, and $d'(e) \coloneqq h(d(e))$.
    Through this change of variables, $S_\delta^{(i)}(X)$ can be seen in terms of the magnitude of a persistence module $M$ satisfying
    \begin{equation*}
        M \cong \bigoplus_{\substack{e \in PH_i(X_\bullet) \\ |e| >\delta}} \mathbf{k} \left[d'(e), b'(e) \right),
    \end{equation*}
     that is, this magnitude is given by
    \begin{equation*}
        |M| = S_\delta^{(i)}(X) = \sum_{\substack{e \in PH_i(X_\bullet) \\ |e| >\delta}} \left\{ \exp\left(-d'(e)\right) - \exp\left(-b'(e)\right) \right\}
    \end{equation*}
    in the sense of Govc-Hepworth \cite{GH21}.
\end{rem}  

\begin{dfn}
    If the limit
    \begin{equation*}
        \beta_i^{\rm phf}(X) \coloneqq \lim_{\delta \searrow 0} \frac{S_\delta^{(i)}(X)}{|\log \delta|} 
    \end{equation*}
    exists, we call $\beta_i^{\rm phf}(X)$ the $i$-th average $ph$-fractal Betti number. 
    If $\beta_i^{\rm phf}(X)$ can be defined for all $i$, we define the average $ph$-fractal Euler number by
    \begin{equation*}
        \chi_{a}^{\rm phf}(X) \coloneqq \sum_{i=0}^{d} (-1)^i \beta_i^{\rm phf}(X).
    \end{equation*}
    \label{def:ph-fractal-betti-number}
\end{dfn}

We provide an useful lemma for computation of examples in the following section.

\begin{lem}
  For a compact subset $X$ of $\mathbb{R}^d$, assume $\sigma_i \neq 0$. 
  Then, the following properties are equivalent. 
  \begin{enumerate}
    \item There exists a monotonically decreasing sequence $\{a_j\}_{j=1}^\infty$ $(a_j \in \mathbb{R}_{>0})$ converging to $0$ which satisfies 
    \begin{enumerate}
      \item $\displaystyle \lim_{j \rightarrow \infty} \frac{1}{ \left| \log a_j\right|}S_{a_{j+1}}^{(i)}(X) = \beta$; 
      \item There exists a constant $0 < c < 1$ such that $c \cdot a_j \leq a_{j+1} < a_j$ for any $j\in \mathbb{N}$.
    \end{enumerate}
    \item $\beta_i ^{\rm phf} (X)$ exists and equals $\beta$.
  \end{enumerate}
  \label{lem: the existence of average ph fractal Betti number}
\end{lem}

\begin{proof}
    We show that ($1$) implies ($2$). 
    For any sufficiently small $\delta > 0$, there exists a unique $j \in \mathbb{N}$ such that $a_{j+1} \leq \delta < a_j$. 
    By definition, the function $S_{\gamma}^{(i)}(X)$ is non-increasing with respect to $\gamma$; that is, $S_{\gamma}^{(i)}(X) \geq S_{\gamma'}^{(i)}(X)$ holds for $0 < \gamma < \gamma'$
    Furthermore, from condition (b), $c^2 a_{j-1} \leq c a_j \leq a_{j+1}$.
    Thus we get 
    \begin{equation*}
        \frac{1}{ \left| \log c^2 \cdot a_{j-1}\right|}S_{a_{j}}^{(i)}(X) 
        \leq 
        \frac{1}{ \left| \log \delta \right|}S_{\delta}^{(i)}(X)
        \leq 
        \frac{1}{ \left| \log a_j\right|}S_{a_{j+1}}^{(i)}(X).
    \end{equation*}
    Since 
    \begin{equation*}
        \frac{S_{a_{j+1}}^{(i)}(X)}{|\log a_j|} - \frac{S_{a_{j}}^{(i)}(X)}{| \log c^2 \cdot a_{j-1}|} = 
        \frac{S_{a_{j+1}}^{(i)}(X)}{|\log a_j|} - \frac{S_{a_{j}}^{(i)}(X)}{|\log a_{j-1} |} \frac{|\log a_{j-1}|}{|\log a_{j-1} + 2 \log c|} \notag \\
        \to 0,
    \end{equation*}
    we obtain
    \begin{equation*}
        \left| \frac{S_\delta ^{(i)}(X)}{|\log \delta |} - \beta \right| \to 0.
    \end{equation*}
    \indent
    The converse is obvious. 
    \qed
\end{proof}
\vspace{1em}
As is clear from the definition, under 'good' conditions where the average fractal Euler number $\chi_{f}^{a}$ of Llorente-Winter\cite{LW07} is determined, it agrees with our average $ph$-fractal Euler number. 

\begin{prop}
    \label{prp: the relation between phf Euler number and average fractal Euler number}
    Let $X$ be a compact subset of $\mathbb{R}^d$. 
    Assume that $\mathfrak{B} = \{0\}$, the Euler exponent $\sigma > 0$ and average fractal Euler number exist, and each of the PH-complexity $comp_{PH_i} (X)$ equals to $0$ or the Euler exponent. 
    Also assume that $b (e) = 0$ for any barcode $e \in PH_i (X_\bullet)$ and $\delta^{\sigma_i}I_{i,\delta}(X) = o(|\log \delta|)\ (\delta \searrow 0)$ for all $i$. 
    Then it holds that
    \begin{equation*}
        \chi_{f}^{a}(X) = \chi_{a}^{\rm phf} (X).
    \end{equation*}
\end{prop}

In particular, all examples of fractals considered by Llorente-Winter in \cite{LW07} satisfy the above conditions, thus their fractal Euler number coincides with ours. 

\begin{proof}
    By $\mathfrak{B} (X) = \{0\}$, we see
    \begin{equation}
        \chi_{f}^{a} (X)=\lim_{\delta \searrow 0} \sum_{i=0}^{d} \frac{(-1)^i}{|\log \delta|} \int_\delta^1 \left(\frac{\epsilon}{d_\infty}\right)^\sigma b_i(X_\epsilon) \frac{d\epsilon}{\epsilon}. 
        \label{eq: decomposition of Euler number}
    \end{equation}
    Moreover, since $PH_\ast (X_\bullet)$ with the filtration indexed by $\mathbb{R}_{>0}$ is pointwise finite-dimensional, it admits a barcode decomposition.
    Then, since $b(e) = 0$ for any barcode $e \in PH_i (X_\bullet)$, the integral part of Equation \ref{eq: decomposition of Euler number} can be expressed as follows: 
    \begin{equation*}
        A_{i,\delta} (X) \coloneqq
        \int_\delta^1 \left(\frac{\epsilon}{d_\infty}\right)^\sigma b_i(X_\epsilon) \frac{d\epsilon}{\epsilon}
        = \frac{1}{\sigma} \sum_{\substack{e\in PH_i (X_\bullet)\\ d(e) > \delta}} \left\{ \left( \frac{d(e)}{d_\infty} \right)^{\sigma} - \left( \frac{\delta}{d_\infty} \right)^{\sigma}\right\}
    \end{equation*}
    Here, we note that $|e|>\delta$ and $d(e) >\delta$ are the same because $b(e) = 0$. 
    Then, in the case of $\sigma_i =\sigma$, we have by the assumption
    \begin{align*}
        \frac{S_\delta^{(i)} (X) - A_{i,\delta}(X)}{|\log \delta|} 
        &= \frac{1}{\sigma|\log\delta|} \sum_{\substack{e\in PH_i (X)\\ d(e) > \delta}} \left( \frac{\delta}{d_\infty} \right)^{\sigma}\\
        &= \frac{1}{\sigma d_\infty^\sigma}\frac{\delta^\sigma I_{i,\delta} (X)}{|\log\delta|} \rightarrow 0 \quad (\delta \searrow 0).
    \end{align*}
    In the case of $\sigma_i = 0$, 
    it means that $A_{i,\delta} (X)$ is bounded as $\delta \searrow 0$, thus $\displaystyle \lim_{\delta \searrow 0} \frac{A_{i,\delta} (X)}{|\log \delta|} = 0$.
    Furthermore, $S_\delta^{(i)} (X) = 0$ by the definition. 
    Therefore, 
    \begin{equation*}
        \lim_{\delta \searrow 0}\frac{1}{|\log \delta|} \int_\delta^1 \left(\frac{\epsilon}{d_\infty}\right)^\sigma b_i(X_\epsilon) \frac{d\epsilon}{\epsilon}
        = \beta_i ^{\rm phf} (X)
    \end{equation*}
    holds for all $i$.
    \qed
\end{proof}

\section{Example \I}

\subsection{Cantor set on the line}
    The Cantor set $C$ is a classical fractal constructed as illustrated in figure \ref{fig: Cantor set}. 
    First, we divide the interval $[0,1]$ into three and remove the open middle third $(\frac{1}{3}, \frac{2}{3})$. 
    Subsequently, the remaining intervals, $\left[0, \frac{1}{3}\right]$ and $\left[\frac{2}{3},1\right]$, are each divided into three, from which the middle thirds are similarly removed. 
    The Cantor set $C$ is the self-similar set obtained by iterating this process infinitely many times. 
    Then, $PH_0 (C)$ consists of $2^{i-1}$ barcodes $e$ of the form
    \begin{equation*}
        \left(b(e), d(e)\right) = \left(0,\frac{1}{6}\left( \frac{1}{3}\right)^{i-1} \right).
    \end{equation*}

    \begin{figure}[hbt]
    \centering
    \begin{tikzpicture}
        \begin{scope}[scale=1]
        \draw[thick] (0,0)--(9,0);

        \draw[thick] (0,-1/3)--(3,-1/3);
        \draw[thick] (6,-1/3)--(9,-1/3);

        \draw[thick] (0,-2/3)--(1,-2/3);
        \draw[thick] (2,-2/3)--(3,-2/3);
        \draw[thick] (6,-2/3)--(7,-2/3);
        \draw[thick] (8,-2/3)--(9,-2/3);

        \draw[thick] (0,-1)--(1/3,-1);
        \draw[thick] (2/3,-1)--(1,-1);
        \draw[thick] (2,-1)--(2+1/3,-1);
        \draw[thick] (2+2/3,-1)--(3,-1);
        \draw[thick] (6,-1)--(6+1/3,-1);
        \draw[thick] (6+2/3,-1)--(7,-1);
        \draw[thick] (8,-1)--(8+1/3,-1);
        \draw[thick] (8+2/3,-1)--(9,-1);
        \foreach \P in {4.5}{
            \foreach \Q in {1/6, 1/3, 1/2}{
            \fill[black] (\P, -7/6-\Q) circle (1pt);
            }
        }
        \end{scope}
    \end{tikzpicture}
    \caption{Cantor set $C$}
    \label{fig: Cantor set}
    \end{figure}
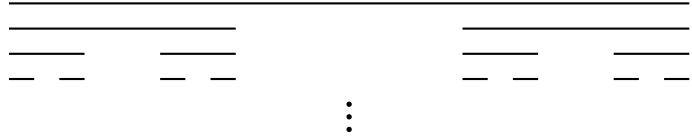
    Hence, the $0$th PH-complexity $\sigma_0$ is 
    \begin{equation*}
    \sigma_0 = \frac{\log 2}{\log 3}. 
    \end{equation*}
    Set $a_j = \frac{1}{2} \left( \frac{1}{3} \right)^j$ ($j\in \N$). 
    This sequence satisfies the condition 2(b) of Lemma \ref{lem: the existence of average ph fractal Betti number}. 
    Then, since the diameter of $C$ is $1$, $S_{a_{j+1}}^{(0)} (C)$ is computed as follows:
    \begin{align*}
    S_{a_{j+1}}^{(0)} (C) &= \frac{1}{\sigma_0}\left\{1 + \sum_{i=1}^{j} \left\{\frac{1}{6} \left(\frac{1}{3}\right)^{i-1}\right\}^{\sigma_0} 2^{i-1} \right\}\\
    &= \frac{1}{\sigma_0} \left\{ 1+ \left(\frac{1}{6}\right)^{\sigma_0} j \right\}
    \end{align*}
    Therefore, the limit is given by
    \begin{equation*}
        \beta_0^{\rm phf}(C) = \lim_{j \rightarrow \infty}\frac{S_{a_{j+1}}^{(0)}(C)}{ \left| \log a_j\right|} = \frac{2^{-(\sigma_0+1)}}{\log2} = 0.466\dots. 
    \end{equation*}
    By definition, $\chi_{a}^{\rm phf}(C) = \beta_0^{\rm phf} (C)$, and it coincides with average fractal Euler number $\chi_f^{a}(C)$ computed in Example$2.11$ in Llerente-Winter \cite{LW07}.

\subsection{Sierpi\'{n}ski carpet}
    The Sierpinski carpet $SC$ is a fractal constructed as figure \ref{fig: Sierpinski carpet}. 
    First, we divide the square $[0,1]\times[0,1]$ into nine and remove the central open square $\left(\frac{1}{3}, \frac{2}{3}\right) \times \left(\frac{1}{3}, \frac{2}{3}\right)$. 
    Subsequently, for each of the remaining $8$ squares, we divide the square into nine and removed the central square. 
    The Sierpinski carpet $SC$ is the self-similar set obtained by iterating this process infinitely many times. 
    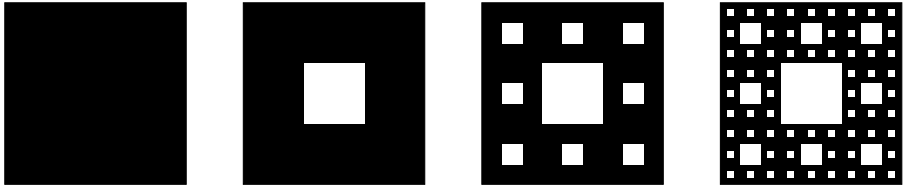
\begin{figure}[hbt]
    \centering
    \begin{minipage}{0.2\textwidth}
        \begin{tikzpicture}
        \begin{scope}[scale=4/15]
            \filldraw [black] (0,0) rectangle (9,9);
        \end{scope}
        \end{tikzpicture}
    \end{minipage}
    \begin{minipage}{0.2\textwidth}
        \begin{tikzpicture}
        \begin{scope}[scale=4/15]
            \filldraw [black] (0,0) rectangle (9,9);
            \fill [white] (3,3) rectangle (6,6);
        \end{scope}
        \end{tikzpicture}
    \end{minipage}
    \begin{minipage}{0.2\textwidth}
        \begin{tikzpicture}
        \begin{scope}[scale=4/15]
            \filldraw [black] (0,0) rectangle (9,9);
            \fill [white] (3,3) rectangle (6,6);
            \foreach \p in {1,4,7}{
            \foreach \q in {1,4,7}{
                \fill [white] (\p,\q) rectangle (\p+1, \q+1);
            }
            }
        \end{scope}
        \end{tikzpicture}
    \end{minipage}
    \begin{minipage}{0.2\textwidth}
        \begin{tikzpicture}
        \begin{scope}[scale=4/15]
            \filldraw [black] (0,0) rectangle (9,9);
            \fill [white] (3,3) rectangle (6,6);
            \foreach \p in {1,4,7}{
            \foreach \q in {1,4,7}{
                \fill [white] (\p,\q) rectangle (\p+1, \q+1);
            }
            }
            \foreach \p in {1/3,4/3,7/3, 1/3 + 3,4/3+3,7/3+3, 1/3+6,4/3+6,7/3+6}{
            \foreach \q in {1/3,4/3,7/3, 1/3 + 3,4/3+3,7/3+3, 1/3+6,4/3+6,7/3+6}{
                \fill [white] (\p,\q) rectangle (\p+1/3, \q+1/3);
            }
            }
        \end{scope}
        \end{tikzpicture}
    \end{minipage}
    \caption{Sierpinski carpet $SC$}
    \label{fig: Sierpinski carpet}
    \end{figure}
    Since the set $SC$ is path connected, the $0$th PH-complexity of $SC$ is $0$.
    Then, the $0$th average $ph$-fractal Betti number of $SC$ is also $0$ by our definition.\\
    \indent
    Next, we think about the $1$st average $ph$-fractal Betti number $\beta_1 ^{\rm phf} (SC)$. 
    Note that $PH_1 (SC)$ consists of $8^{i-1}$ barcodes of the form
    \begin{equation*}
    \left( 0, \frac{1}{6} \left( \frac{1}{3}\right)^{i-1}\right).
    \end{equation*}
    Then, the $1$st PH-complexity of $SC$ is 
    \begin{equation*}
        \sigma_1 = \frac{\log 8}{\log 3}. 
    \end{equation*}
    Set $a_j = \left(\frac{1}{3}\right)^j \frac{1}{2}$. 
    This sequence satisfies the condition 2(b) of Lemma \ref{lem: the existence of average ph fractal Betti number}. 
    Since the diameter of $SC$ is $\sqrt{2}$, $S_{a_{j+1}}^{(1)} (SC)$ is computed as follows:
    \begin{align*}
    S_{a_{j+1}}^{(1)} (SC) = \frac{1}{\sigma_1}\sum_{i=1}^{j} \left\{\frac{1}{6\sqrt{2}}\left(\frac{1}{3}\right)^{i-1}\right\}^{\sigma_1} 8^{i-1}
    = \left(\frac{1}{6\sqrt{2}}\right)^{\sigma_1} j
    \end{align*}
    Therefore, the limit is given by
    \begin{equation*}
        \beta_1^{\rm phf} (SC) = \lim_{j \rightarrow \infty}\frac{S_{a_{j+1}}^{(1)}(SC)}{ \left| \log a_j\right|} = \frac{1}{ \log 8}\left(\frac{1}{6\sqrt{2}}\right)^{\frac{\log 8}{\log3}} = 0.0084\dots. 
    \end{equation*}
    Since $\beta_0^{\rm phf}(SC) =0$, the average $ph$-fractal Euler number equals $-\beta_1^{\rm phf}(SC)$, and it coincides with $\chi_f^{a} (SC) = -\frac{u^s}{\log 8}$ computed in Example $2.7$ in Llorente-Winter \cite{LW07} (but an incorrect approximation is written there).

\section{Example \II}
Next, we calculate the average $ph$-fractal Betti number for the case where the average fractal Euler number cannot be defined. 
Specifically, we consider the Cantor dust $C\times C$ and the Menger sponge $M$ below.

\subsection{Cantor dust in the plane}
    The Cantor dust $C \times C$ is a fractal given by the product of two Cantor sets as seen in figure \ref{fig: const. of Cantor dust}.
    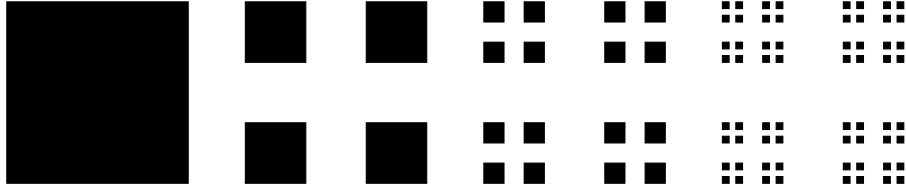
\begin{figure}[hbt]
    \centering
    \begin{minipage}{0.2\textwidth}
        \begin{tikzpicture}
        \begin{scope}[scale=4/15]
            \filldraw [black] (0,0) rectangle (9,9);
        \end{scope}
        \end{tikzpicture}
    \end{minipage}
    \begin{minipage}{0.2\textwidth}
        \begin{tikzpicture}
        \begin{scope}[scale=4/15]
            \foreach \p in {0,6}{
            \foreach \q in {0,6}{
                \filldraw [black] (\p,\q) rectangle (\p+3, \q +3);
            }
            }
        \end{scope}
        \end{tikzpicture}
    \end{minipage}
    \begin{minipage}{0.2\textwidth}
        \begin{tikzpicture}
        \begin{scope}[scale=4/15]
            \foreach \p in {0,2,6,8}{
            \foreach \q in {0,2, 6 ,8}{
                \filldraw [black] (\p,\q) rectangle (\p+1, \q +1);
            }
            }
        \end{scope}
        \end{tikzpicture}
    \end{minipage}
    \begin{minipage}{0.2\textwidth}
        \begin{tikzpicture}
        \begin{scope}[scale=4/15]
            \foreach \p in {0,2/3,2,2+2/3,6,6+2/3,8,8+2/3}{
            \foreach \q in {0,2/3,2,2+2/3,6,6+2/3,8,8+2/3}{
                \filldraw [black] (\p,\q) rectangle (\p+1/3, \q +1/3);
            }
            }
        \end{scope}
        \end{tikzpicture}
    \end{minipage}
    \caption{Cantor dust $C \times C$}
    \label{fig: const. of Cantor dust}
    \end{figure}

    The Cantor dust contains configurations resembling parallelly arranged Cantor sets (and their scaled copies), as illustrated on the right side of Figure \ref{fig: cantor dust 1}. 
    Owing to these geometric feature, the $1$st Betti function of the Cantor dust exhibits the behavior shown in Figure \ref{fig: cantor dust Betti curve}; 
    specifically, the Betti number is not determined at $\epsilon = \frac{1}{6}\left(\frac{1}{3}\right)^{i-1}$ for $i \in \N$. 
    Consequently, the average fractal Euler number $\chi_f^a$ cannot be defined for the Cantor dust.\\
    \indent
    First, we calculate the $0$th average $ph$-fractal Betti number for the Cantor dust. 
    Note that $d_\infty (C\times C) = \sqrt{2}$. 
    The basis of $PH_0 (C \times C)$ consists of the barcode $\left(0,\sqrt{2}\right)$ and $3\cdot4^{i-1}$ barcodes of the form
    \begin{equation*}
    \left( 0, \frac{1}{6} \left( \frac{1}{3}\right)^{i-1}\right).
    \end{equation*}
    Then, the $0$th PH-complexity of $C \times C$ is $\sigma_0 = \frac{\log 4}{\log 3}$. 
    Therefore, by Lemma \ref{lem: the existence of average ph fractal Betti number} with setting $a_j = \frac{1}{6}\left(\frac{1}{3}\right)^{j-1}$, $\beta_0^{\rm phf} (C \times C)$ is calculated as
    \begin{equation}
        \beta_0 ^{\rm phf} (C \times C) = \frac{3}{\log 4}\left(\frac{1}{6\sqrt{2}}\right)^{\frac{\log 4}{\log 3}} = 0.1456\dots.
        \label{eq: cantor dust 0th}
    \end{equation} 

    \begin{figure}[bt]
        \centering
        \includegraphics[width=\textwidth]{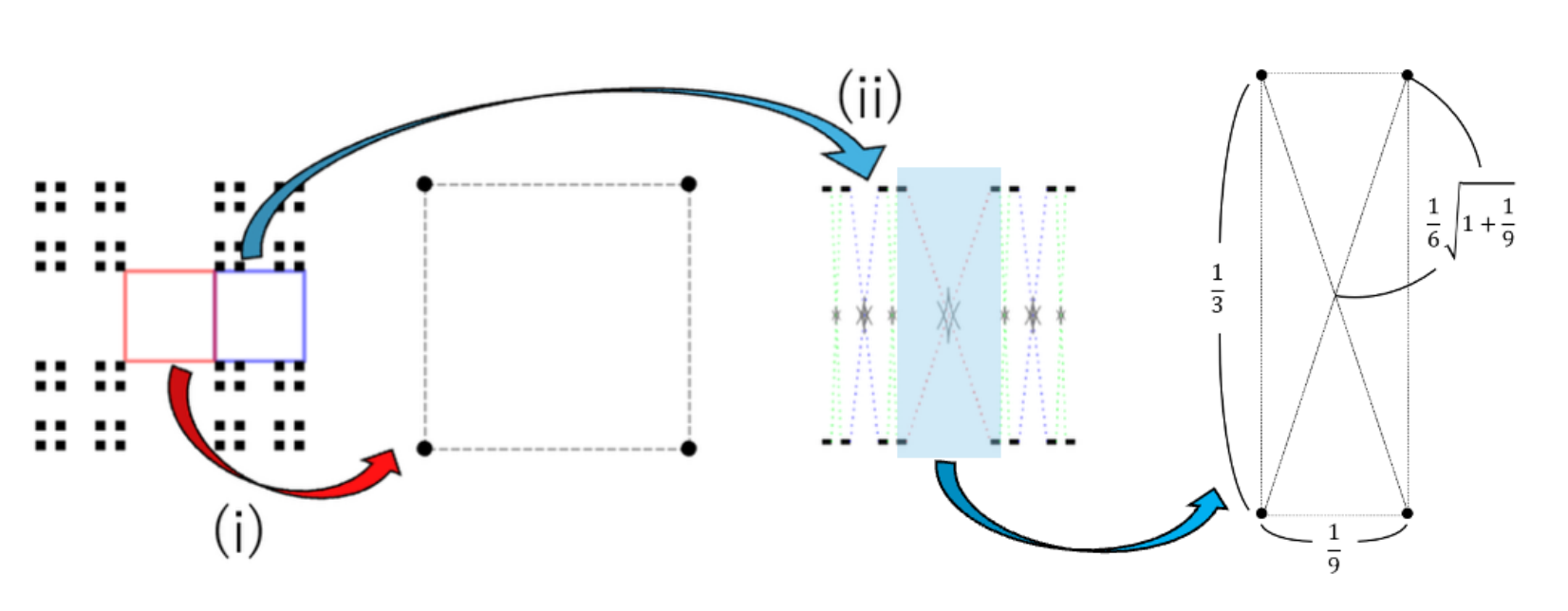}
        \caption{}
        \label{fig: cantor dust 1}
    \end{figure}
    \indent
    Next, we consider the $1$st average $ph$-fractal Betti number of $C \times C \ (=X)$. 
    There are two types of generators of $PH_1 (X_\bullet)$. 
    The `principal holes' are indicated by Figure \ref{fig: cantor dust 1}(i) and its iterations, and the `little barcode dust' is Figure \ref{fig: cantor dust 1}(ii) and its iterations. 
    Precisely, those are represented by: 
    \begin{enumerate}
        \item[(i)] $4^{j-1}$ barcodes of the form $\displaystyle \left(\frac{1}{6}\left(\frac{1}{3}\right)^{j-1}, \frac{\sqrt{2}}{6}\left(\frac{1}{3}\right)^{j-1} \right)$  ($j \in \mathbb{N}$), and \label{item: cantor dust principal}
        \item[(ii)] $4\cdot 4^{j-1}\cdot2^{i-1}$ barcodes of the form \small{$\displaystyle \left(\frac{1}{6}\left(\frac{1}{3}\right)^{j-1}, \frac{1}{6}\sqrt{1+\frac{1}{3^{2i}}}\left(\frac{1}{3}\right)^{j-1} \right)$  ($i, j \in \mathbb{N}$)}. \label{item: cantor dust dust}
    \end{enumerate}
    The $1$st appearance of little barcode dust is caused by barcodes of the form (ii) with $j=1$, (i.e. $1$st step); there are countably many barcodes indexed by $i=1,2,\dots$ and then we put the lifetime
    \begin{equation*}
        l_i = \frac{1}{6} \left(\sqrt{1+\frac{1}{3^{2i}}}-1\right). 
    \end{equation*}
    As shown in Figure \ref{fig: Betti iteration}, the barcodes in $I_j$ ($j \geq 2$) consist of the iteration of the barcode in $I_1$.\\
    \begin{figure}[bt]
        \includegraphics[width=9cm]{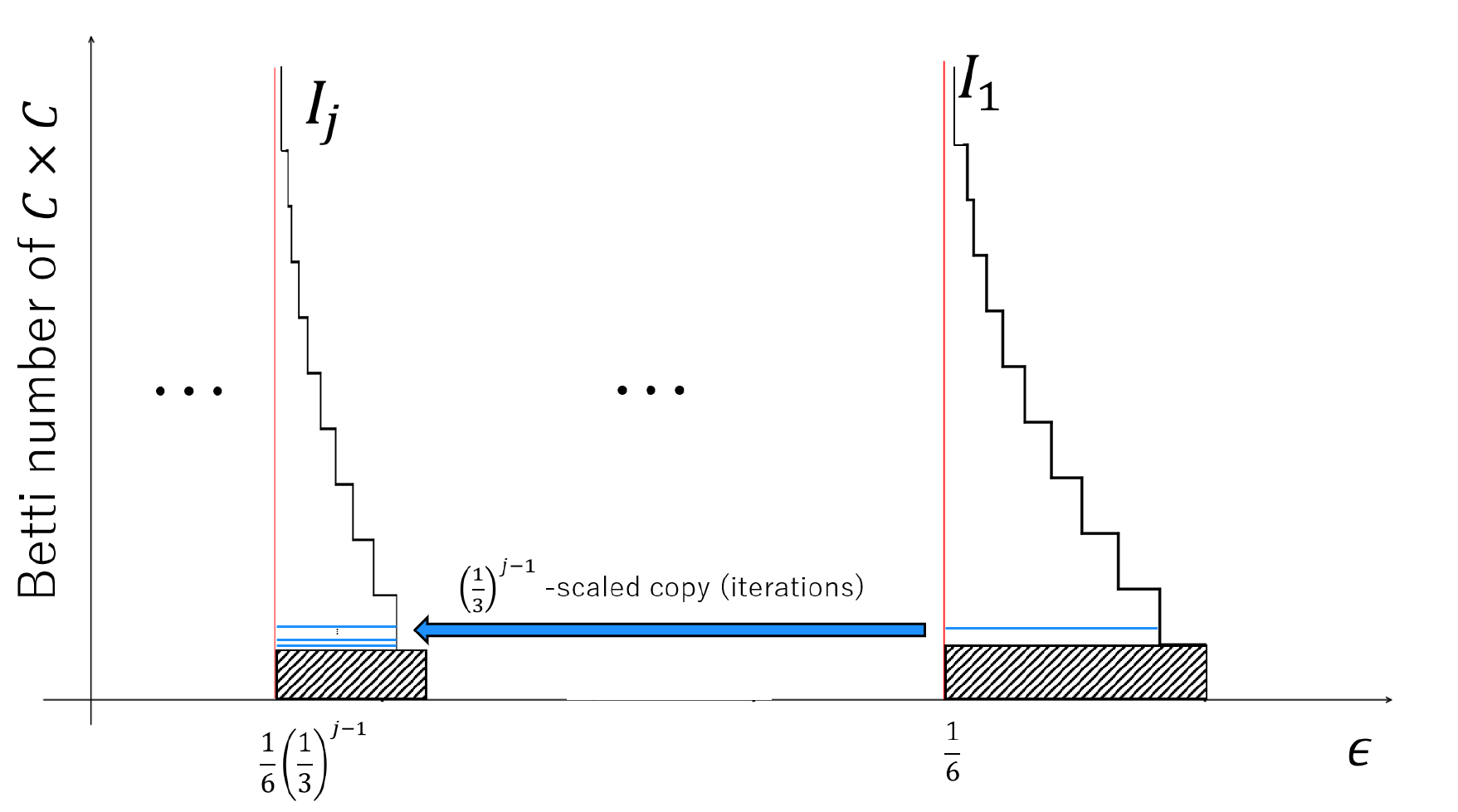}
        \caption{}
        \label{fig: Betti iteration}
    \end{figure}
    \indent
    The $1$st PH-complexity is $\sigma_1 = \sigma_0 = \frac{\log 4}{\log 3}$. 
    In fact, the growth rate of barcodes controlled by $i$ of the form (ii) does not affect the overall growth rate. 
    For any $\epsilon>0$, take $i$ such that $l_i > \epsilon$. 
    From the form of $l_i$, the growth rate of the number of $e$ such that $l_i = |e| > \epsilon$ as $\epsilon \searrow 0$ is less than $\frac{\log 4}{\log 3}$. 
    Furthermore, the growth rate of the number of scaled copies of a barcode $e$ with $|e|=l_i$ is $\frac{\log 4}{\log 3}$. \\
    \indent
    Otherwise, we may use Proposition \ref{prp: equivalent definition of PH-complexity} to directly compute $\sigma_1$. 
    It is clear that $\frac{\log 4}{\log 3} \leq \sigma_1$, so we check $\sigma_1 \leq \frac{\log 4}{\log 3}$.
    For any $\alpha > \frac{\log 4}{\log 3}$, we have $0 < \left(\frac{1}{3}\right)^\alpha 4 < 1$. 
    Thus, since $\sqrt{1+\frac{1}{3^{2i}}} - 1 < \frac{1}{3^i}$, there exists $M \in R$ satisfying; 
    \begin{multline*}
        \sum_{e\in PH_1(C\times C)}|e|^\alpha = 
        \sum_{j=1}^{\infty} \left(\frac{\sqrt{2}}{6} \left(\frac{1}{3}\right)^{j-1} - \frac{1}{6}\left(\frac{1}{3}\right)^{j-1}\right)^\alpha 4^{j-1} \\
        +
        \sum_{j=1}^{\infty} \sum_{i=1}^{\infty} \left(\frac{1}{6}\sqrt{1+\frac{1}{3^{2i}}} \left(\frac{1}{3}\right)^{j-1} - \frac{1}{6}\left(\frac{1}{3}\right)^{j-1}\right)^\alpha 4 \cdot 4^{j-1} \cdot 2^{i-1}
        < M.
    \end{multline*}
    Hence $\sigma_1 \leq \frac{\log 4}{\log 3}$ holds, and we have $\sigma_1 = \frac{\log 4}{\log 3}$.\\
    \indent
    Let $P_1$ and $P_2$ denote the sets of barcodes of the former and latter forms (i) and (ii), respectively. 
    $S_\delta^{(1)}(C\times C)$ is separated to two parts; let $S_\delta^{(1), P_1}$ be the sum running over $e \in P_1$ and $S_\delta^{(1), P_2}$ be the sum running over $e\in P_2$ ($|e|>\delta$).\\
    \indent
    For the contribution from the form (i), apply Lemma \ref{lem: the existence of average ph fractal Betti number} with $a_j = \frac{\sqrt{2}-1}{6}\left(\frac{1}{3}\right)^{j-1}$ ($j \in \N$), then we have 
    \begin{equation}
        \lim_{\delta \searrow 0} \frac{S_\delta^{(1), P_1}}{\left|\log \delta\right|} = \frac{1}{\sigma_1 \log 3}\left(\frac{1}{6\sqrt{2}}\right)^{\sigma_1}\left(2^{\frac{\sigma_1}{2}} - 1\right).
        \label{eq: Cantor dust principal}
    \end{equation}

    \begin{figure}[bt]
        \centering
        \includegraphics[width= 9cm]{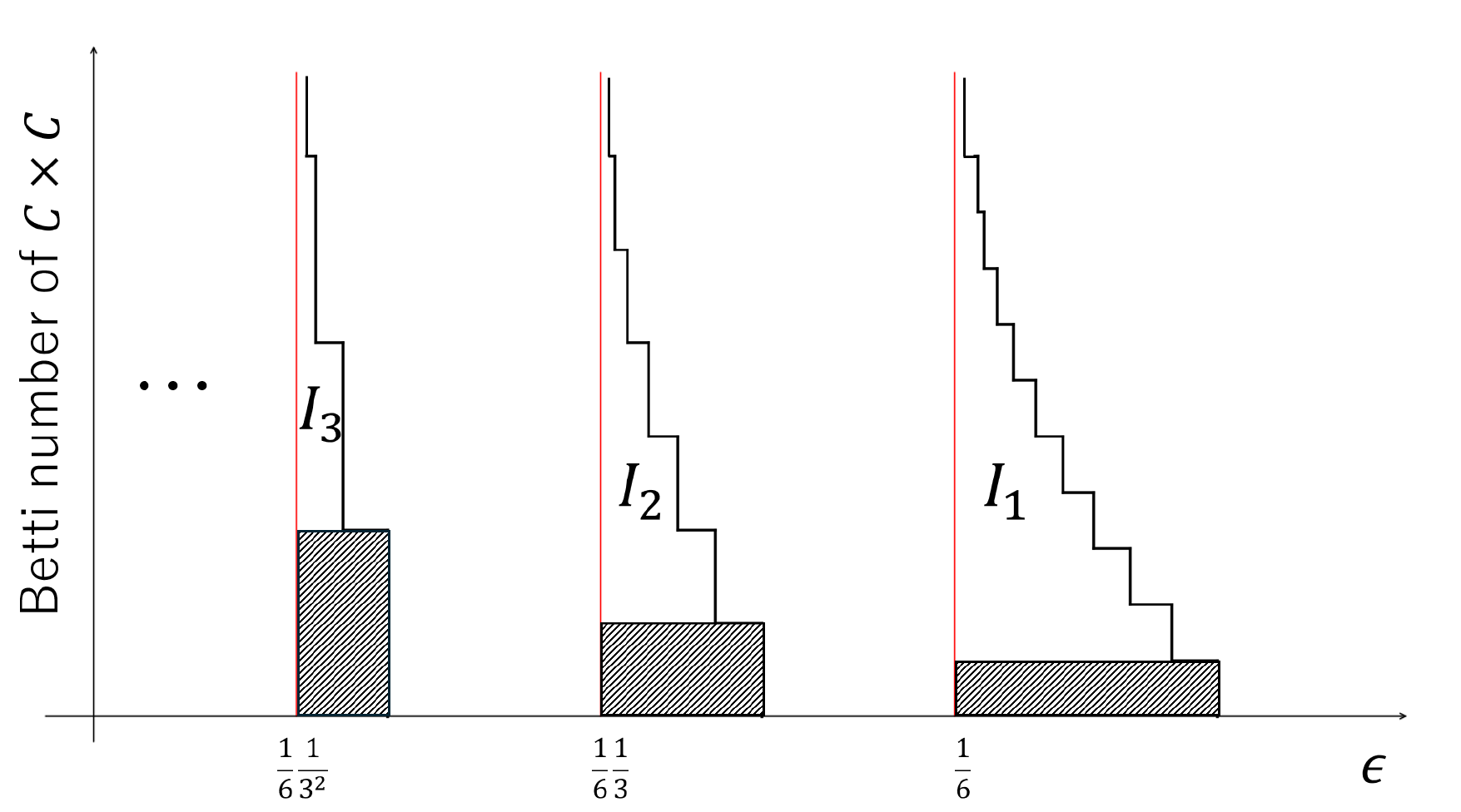}
        \caption{The $1$st Betti curve of $C \times C$}
        \label{fig: cantor dust Betti curve}
    \end{figure}

    \indent
    Next, consider the barcodes of form (ii). 
    Let $I_j$ be integration of the function $\sum_{e} \epsilon^{\sigma_1 - 1}\jeden_e$ where $e$ runs over all barcodes of the form (ii) at the $j$-step, as illustrated in Figure \ref{fig: cantor dust Betti curve}. 
    It is computed as
    \begin{equation*}
        I_j = \frac{1}{\sigma_1}\sum_{i=1}^{\infty} \left(\frac{1}{6\sqrt{2}}\right)^{\sigma_1}\left( \frac{1}{3}\right)^{\sigma_1(j-1)} \allowbreak \left\{\left(1+ \frac{1}{3^{2i}}\right)^{\frac{\sigma_1}{2}}-1\right\}4\cdot 4^{j-1} \cdot 2^{i-1}.
    \end{equation*}
    Since $\left(\frac{1}{3}\right)^{\sigma_1} = \frac{1}{4}$, every $I_j$ is the same amount ($=I_1$). 
    Furthermore, since $0< \frac{\sigma_1}{2} <1$, we see
    \begin{align*}
      I_1 &= \frac{1}{\sigma_1}\sum_{i=1}^{\infty} \left(\frac{1}{6\sqrt{2}}\right)^{\sigma_1}\left\{\left(1+\frac{1}{3^{2i}}\right)^{\frac{\sigma_1}{2}} - 1\right\}4 \cdot2^{i-1}\\
        &< \frac{1}{\sigma_1}\sum_{i=1}^{\infty} \left(\frac{1}{6\sqrt{2}}\right)^{\sigma_1} \left\{ 1 + \left(\frac{1}{3^{2i}}\right)^\frac{\sigma_1}{2} - 1\right\}4 \cdot2^{i-1}
       = \frac{2}{\sigma_1}\left(\frac{1}{6\sqrt{2}}\right)^{\sigma_1}.
    \end{align*}
    Set $a_j = l_1 \left(\frac{1}{3}\right)^{j-1}$. 
    By definition, $S_{a_{j+1}}^{(1),P_2}$ is the sum up to $\left(\frac{1}{3}\right)^{j-1}$-scaled copies $e$ of the form (ii) with $|e| > a_{j+1}$. 
    Thus, it is bounded from above as 
    \begin{equation*}
      S_{a_{j+1}}^{(1),P_2} \leq I_1 + \dots + I_j = jI_1. 
    \end{equation*}
    \indent
    We estimate $S_{a_{j+1}}^{(1),P_2}$ from below. 
    Since
    $
        \frac{1}{9}\left(\sqrt{1+\frac{1}{3^{2j}}}-1\right) < \sqrt{1+\frac{1}{3^{2(j+1)}}}-1 \ (j \in \N), 
    $ 
    we see
    \begin{equation*}
        a_{j+1} < a_j = l_1 \left(\frac{1}{3}\right)^{j-1} < l_2\left(\frac{1}{3}\right)^{j-3} < l_3 \left(\frac{1}{3}\right)^{j-5} < \dots.
    \end{equation*} 
    Fix $j$ and choose $i$ ($= 0, 1, \dots, \left\lfloor \frac{j}{2} \right\rfloor$). 
    For $1 \leq k \leq j -2i$, every $\left(\frac{1}{3}\right)^{k-1}$-scaled copy of a barcode $e$ with $|e| = l_i$ contributes to $S_{a_{j+1}}^{(1),P_2}$. 
    In fact, in the case of $i=2$, we can see that $\left(\frac{1}{3}\right)^{k-1}$-scaled copies ($k=1,2,\dots, j-4$) contribute to $S_{a_{j+1}}^{(1),P_2}$ by the inequality $a_{j+1} < l_3\left(\frac{1}{3}\right)^{j-5}$. 
    Let $J_i$ be the contribution to $S_{a_{j+1}}^{(1),P_2}$ from all scaled copies of barcodes with the lifetime $l_i$. 
    Then a cancellation happens and we get 
    \begin{equation*}
        J_i = \frac{j-2i}{\sigma} \left(\frac{1}{6\sqrt{2}}\right)^{\sigma_1} \left\{\left(1+\frac{1}{3^{2i}}\right)^{\frac{\sigma_1}{2}}-1\right\}4\cdot 2^{i-1}.
    \end{equation*}
    Thus, $S_{a_{j+1}}^{(1),P_2}$ can be bounded from below as follows:
    \begin{equation*}
        \sum_{i=1}^{\lfloor\frac{j}{2} \rfloor} J_i = \sum_{i=1}^{\lfloor\frac{j}{2} \rfloor} \frac{j-2i}{\sigma_1} \left(\frac{1}{6\sqrt{2}}\right)^{\sigma_1} \left\{\left(1+\frac{1}{3^{2i}}\right)^{\frac{\sigma_1}{2}}-1\right\}4\cdot 2^{i-1}
        \leq S_{a_{j+1}}^{(1),P_2}.
    \end{equation*}

    Consequently, we obtain the following inequality:
    \begin{equation*}
        \sum_{i=1}^{\lfloor\frac{j}{2} \rfloor} J_i \leq S_{a_{j+1}}^{(1),P_2} \leq {j I_1}.
    \end{equation*}
    If $j$ is even, we have
    \begin{align*}
        jI_1 - \sum_{i=1}^{\lfloor\frac{j}{2} \rfloor} J_i 
        &= \frac{1}{\sigma_1} 2\left|\sum_{l=1}^{\frac{j}{2}}\left\{ \sum_{i=l+1}^{\infty} \left(\frac{1}{6\sqrt{2}}\right)^{\sigma_1} \left\{\left(1+\frac{1}{3^{2i}}\right)^{\frac{\sigma_1}{2}} - 1\right\}4 \cdot 2^{i-1} \right\}\right| \\
        &< \frac{1}{\sigma_1} 2 \left(\frac{1}{6\sqrt{2}}\right)^{\sigma_1} \sum_{l=1}^{\frac{j}{2}} \sum_{i=l+1}^{\infty} \left(\frac{1}{2}\right)^{i-1}\notag\\
        &< \frac{4}{\sigma_1} \left(\frac{1}{6\sqrt{2}}\right)^{\sigma_1}.
    \end{align*}
    For odd $j$, it is similar. 
    Thus,
    \begin{align*}
        \frac{jI_1 - \sum_{i=1}^{\lfloor\frac{j}{2} \rfloor} J_i}{\left|\log a_j\right|}
        < \frac{4 \left(\frac{1}{6\sqrt{2}}\right)^{\sigma_1}}{\sigma_1 \left| (j-1) \log 3 - \log l_1\right|}
        \rightarrow 0 \quad (j \rightarrow \infty).
    \end{align*}
    By applying Lemma 1 with $a_j = l_1 \left(\frac{1}{3}\right)^{j-1}$ ($j\in \N$), we have
    \begin{equation}
        \lim_{\delta \searrow 0} \frac{S_\delta^{(1), P_2}}{\left|\log \delta\right|} = \frac{I_1}{\log 3}
        \label{eq: Cantor dust dust}
    \end{equation}
    Thus, by (\ref{eq: Cantor dust principal}) and (\ref{eq: Cantor dust dust}), we have
    \begin{equation*}
        \beta_1^{\rm phf} (C \times C) = \frac{1}{\log 4}\left\{\left(\frac{1}{6\sqrt{2}}\right)^{\frac{\log 4}{\log 3}}\left(2^{\frac{\log 2}{\log 3}} - 1\right) + \sigma_1 I_1\right\} = 0.0438\dots.
    \end{equation*}
    Thus, $\chi_a^{\rm phf}(C \times C) = \beta_0^{\rm phf}(C \times C) - \beta_1^{\rm phf} (C \times C)=0.1018\dots$.

\subsection{Menger sponge}
    The Menger sponge $M$ is a fractal known as a 3-dimensional extension of the Sierpinski carpet, and is constructed as shown in Figure \ref{fig: Menger sponge const.}. 
    Since the Menger sponge is path-connected, the $0$th PH-complexity and the $0$th average $ph$-fractal Betti number of $M$ are $0$. \\
    \begin{figure}[hbt]
        \centering
        \includegraphics[width=\textwidth]{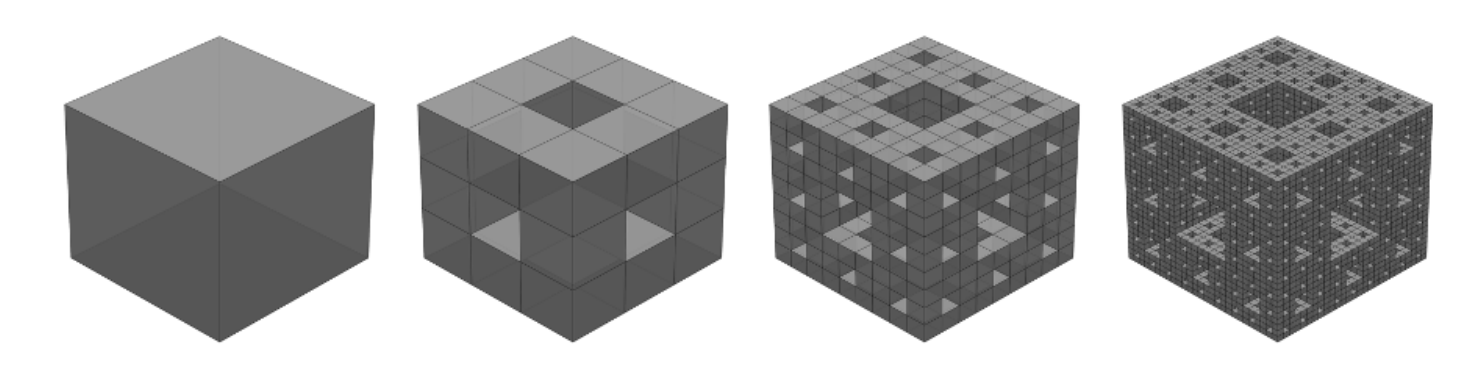}
        \caption{Menger sponge $M$}
        \label{fig: Menger sponge const.}
    \end{figure}
    \indent
    First of all, we consider the $2$nd average $ph$-fractal Betti number of $M$. 
    As for the case of $2$nd homology of $M$, `principal holes' and `little barcode dust' appear like as in the case of $1$st homology for $C\times C$. 
    The `principal holes' are indicated by Figure \ref{fig: Menger sponge kernel}(i) and its iterations, and the `little barcode dust' is Figure \ref{fig: Menger sponge kernel}(ii) and its iterations. 
    Precisely, those are represented by: 
    \begin{itemize}
      \item[(i)] $20^{j-1}$ barcodes of the form $\left( \frac{1}{6}\left(\frac{1}{3}\right)^{j-1}, \frac{\sqrt{2}}{6}\left(\frac{1}{3}\right)^{j-1}\right)$ ($j \in \N$), and
      \item[(ii)] $6 \cdot 20^{j-1}\cdot 2^{i-1}$ barcodes of the form$\left( \frac{1}{6}\left(\frac{1}{3}\right)^{j-1}, \frac{1}{6}\sqrt{1 + \frac{1}{3^{2i}}}\left(\frac{1}{3}\right)^{j-1}\right)$ ($i,j \in \N$).
    \end{itemize}

    \indent
    From this information, the $2$nd PH-complexity of $M$ is given as $\sigma_2 = \frac{\log 20}{\log 3}$. 
    This is because, as in the case of $C \times C$, the growth rate controlled by $i$ in the little barcode dust does not affect the overall growth rate.\\
    \begin{figure}[hbt]
        \centering
        \includegraphics[width=0.9\textwidth]{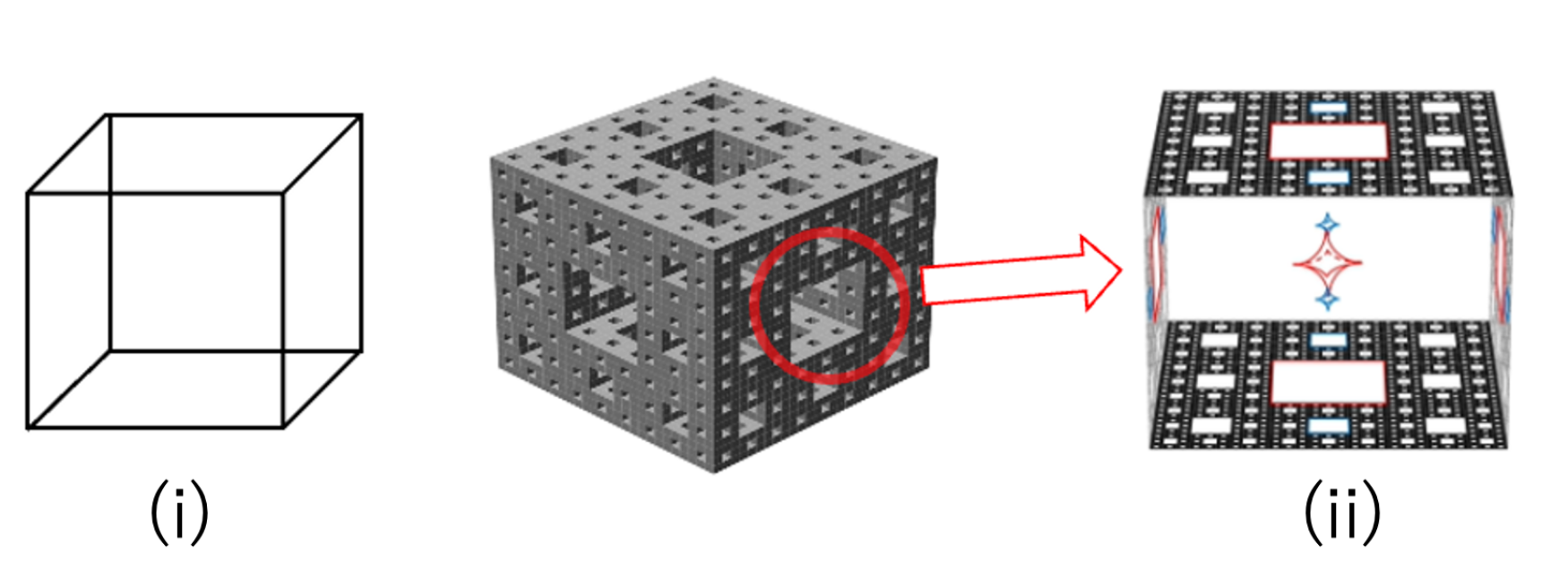}
        \caption{}
        \label{fig: Menger sponge kernel}
    \end{figure}
    \indent
    Similarly as in the case of $C \times C$, the Betti curve of the latter is as shown in Figure \ref{fig: cantor dust Betti curve}, and a similar calculation shows that $I_j$ is constant independent of $j$. 
    Since $\alpha \coloneqq \frac{\sigma_2}{2} = 1.36\dots$, we see that $1+2x -(1+x)^\alpha > 0$ for $0<x<1$. 
    Therefore, $I_1$ is finite, since
    \begin{align*}
        I_1 &= \frac{1}{\sigma_2}\sum_{j=1}^\infty \left\{\left(\frac{1}{6\sqrt{3}}\right)^{\sigma_2}\left(1+\frac{1}{3^{2j}}\right)^{\frac{\sigma_2}{2}} - \left(\frac{1}{6\sqrt{3}}\right)^{\sigma_2}\right\} 6 \cdot2^{j-1}\\
        &< \frac{1}{\sigma_2}\sum_{j=1}^\infty \left(\frac{1}{6\sqrt{3}}\right)^{\sigma_2} \left\{1+\frac{2}{3^{2j}} - 1\right\}6\cdot2^{j-1}
        = \left(\frac{1}{6\sqrt{3}}\right)^{\sigma_2}\frac{12}{7\sigma_2} \qquad (\because \frac{\log 20}{2\log 3} < 2).
    \end{align*}
    By performing a calculation similar to the $1$st average $ph$-fractal Betti number of $C \times C$, we obtain the $2$nd average $ph$-fractal Betti number for $M$:
    \begin{equation*}
        \beta_2^{\rm phf} (M) = \frac{1}{\log 20}\left\{ \left(\frac{1}{6\sqrt{3}}\right)^{\frac{\log 20}{\log 3}}\left(2^{\frac{\log 20}{2 \log 3}} - 1\right) + \sigma_2 I_1 \right\} =  0.001555\dots.
    \end{equation*}
    \indent
    Next, we consider the $1$st average $ph$-fractal Betti number of $M$. 
    $PH_1 (M)$ consists of the barcodes of the form $\left( 0 , \frac{1}{6} \left(\frac{1}{3}\right)^{j-1} \right)$ ($j \in \mathbb{N}$). 
    For $j=1$, the number of barcodes of the form $\left(0,\frac{1}{6}\right)$ is $5$. 
    At the $j$-step, there are $20^{j-1}$ copies of $1$-dimensional cycle in $M$, but some of them are mutually homologous, so we have to reduce the number of those redundant $\left(\frac{1}{3}\right)^{j-1}$-scaled copies. 
    In fact, the number of $1$st Betti number of $M$ is as shown in Table \ref{tab: Betti number of M}.
    \begin{table}[htbp]
        \centering
        \caption{$1$st Betti number of $M$}
        \label{tab: Betti number of M}
        \begin{tabular}{|l|c|c|c|r|}
            $\epsilon$ & $\frac{\sqrt{2}}{6}\frac{1}{3}$ & $\frac{\sqrt{2}}{6}\left(\frac{1}{3}\right)^2$ & $\frac{\sqrt{2}}{6}\left(\frac{1}{3}\right)^3$ &$\frac{\sqrt{2}}{6}\left(\frac{1}{3}\right)^4$\\
            \hline
            $H_1(M_\epsilon)$ & 5 & 76 & 1328 & 25024
        \end{tabular}
    \end{table}

    \begin{figure}[hbt]
        \centering
        \includegraphics[width=\textwidth]{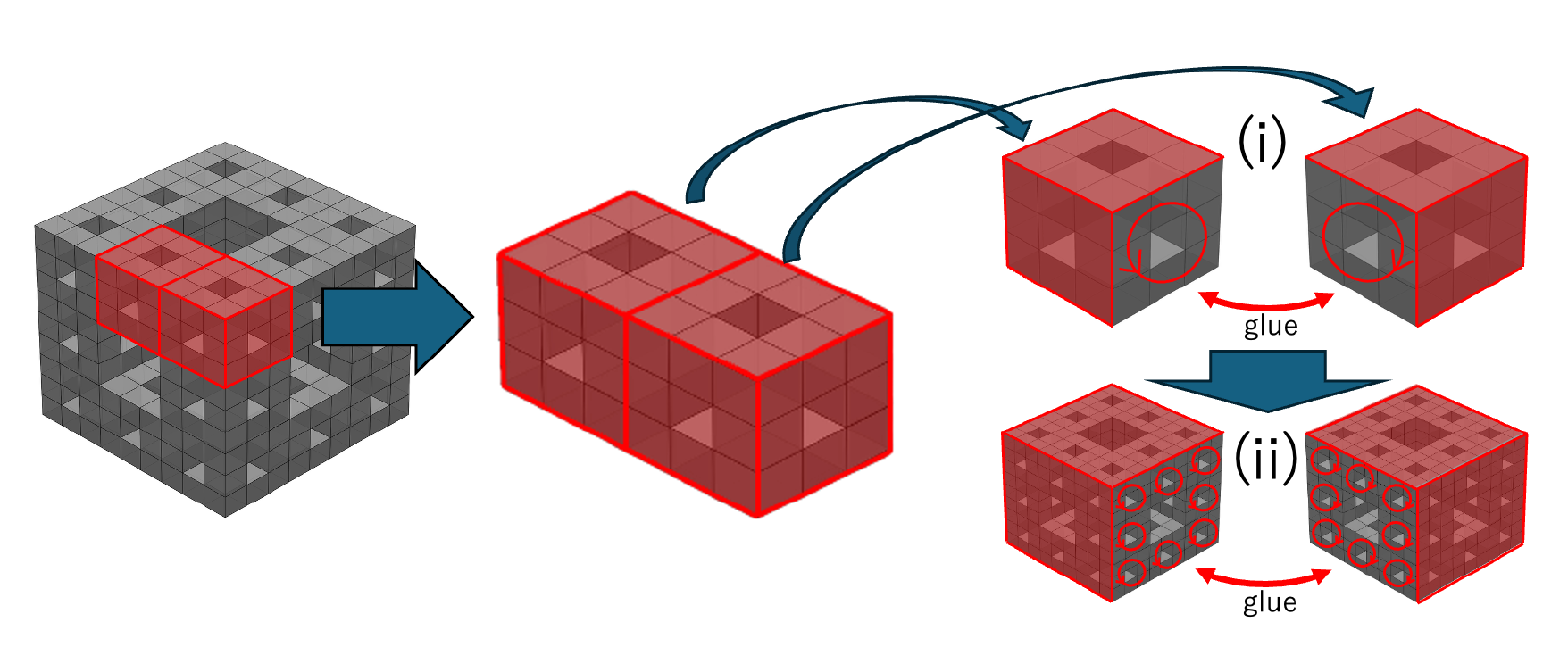}
        \caption{}
        \label{fig: reduction of barcodes}
    \end{figure}
    Let $A_j$ be the number the barcodes of the form $\left( 0 , \frac{1}{6} \left(\frac{1}{3}\right)^{j-1} \right)$, and let $B_j$ be the number of reduced barcodes, that is, $B_j = 5 \cdot 20^{j-1} - A_j$. 
    Within a single cube, there are $24$ faces where $\frac{1}{9}$-cubes (i.e. the length of each edge is $\frac{1}{9}$) touch each other, as shown in Figure \ref{fig: reduction of barcodes}(i). 
    Upon further subdivision, apart from those generated inside the $\frac{1}{9}$-cubes, $8$ forms like \ref{fig: reduction of barcodes}(ii) appear at the parts where the cubes are already in contact.
    Then, we obtain the recurrence relation
    \begin{align*}
        B_1 &= 0, \\
        B_j &= 24\cdot 20^{j-2} + 8 \cdot B_{j-1} \quad (j\geq 2). 
    \end{align*}
    By solving this recurrence relation, we obtain $B_j = -2^{3j-2} + 2 \cdot 20^{j-1}$ and $A_j = 3 \cdot 20^{j-1} + 2^{3j-2}$. 
    From this information, the $1$st PH-complexity of $M$ is given as $\sigma_1 = \frac{\log 20}{\log 3}$. 
    Then, by Lemma \ref{lem: the existence of average ph fractal Betti number}, the $1$st average $ph$-fractal Betti number of $M$ can be calculated as follows:
    \begin{equation*}
        \beta_1^{\rm phf} (M) = \frac{3}{\log 20}\left(\frac{1}{6\sqrt{3}}\right)^{\frac{\log 20}{\log 3}} = 0.001691\dots.
    \end{equation*}
    Thus, $\chi_a^{\rm phf} (M) = \beta_0^{\rm phf} (M) -\beta_1^{\rm phf}(M) + \beta_2^{\rm phf}(M) = -0.0001353\dots$.

\section{Further discussion}
In the present paper, we have modified Llorente-Winter's average fractal Euler number to a broader class of fractals by applying the perspective of persistent homology. \\
\indent
This definition is indebted by the PH-complexity introduced by MacPherson-Schweinhart \cite{MS12}. 
As another quantity to characterize fractals by using persistent homology, Schweinhart\cite{S21} also introduced a better notion of fractal dimension by sampling finite points of $X$; 
\begin{equation*}
    \dim_{PH}^i (X) = \inf \left\{ \alpha \colon \exists \  C >0 \ s.t. \sum_{e\in PH_i (\mathbf{X})} |e|^\alpha < C \ \forall \text{ finite } \mathbf{x} \subset X\right\}. 
\end{equation*}
This kind of definition using finite sampling can also be seen for the {\em magnitude} of compact metric spaces \cite{LM17}. 
Therefore, there should be an alternative better approach to defining fractal Betti numbers by taking a supremum of the same kind of summand as $S_\delta^{(i)}$ but running over finite sampling points from $X$. 
Moreover, Winter and Z\"{a}hle \cite{W06, WZ13} defined not only the (average) fractal Euler number but also curvature currents, that may adapt this framework. 
We will discuss this approach somewhere else.

\end{document}